\documentclass[12pt,en]{elegantpaper}
\usepackage{extarrows}
\usepackage{mathrsfs}
\numberwithin{equation}{section}
\allowdisplaybreaks[4]
\everymath{\displaystyle}

\title{Classification of positive solutions to a class of Laplace equations with a gradient term}

\author{Jingbo Dou \and  Benfeng Shi \and Tian Wu \and Hua Zhu\footnotemark[1]}

\begin{document}

\date{}
\maketitle

\renewcommand{\thefootnote}{\fnsymbol{footnote}}
\newcommand{\Red}[1]{\textcolor{red}{#1}}
\newcommand{\Blue}[1]{\textcolor{blue}{#1}}
\newcommand{\Green}[1]{\textcolor{green}{#1}}

\footnotetext[1]{The corresponding author.}

\begin{abstract}

In this paper, we investigate positive solutions to a class of Laplace equations with a gradient term on a complete, connected, and noncompact Riemannian manifold \((M^n,g)\) with nonnegative Ricci curvature, namely
\[-\Delta u = f(u)|\nabla u|^q\quad\text{in }~M^n,\]
where \(n\geqslant 3\), \(q>0,\) and \(f\) is a positive continuous function. We prove some Liouville theorems employing a key differential identity derived via the invariant tensor technique. In particular, for \(f(u)=u^{\frac{2-q}{n-2}(n+\frac{q}{1-q})-1}\) is the second critical case in dimension \(n=3,4,5\), without any additional conditions, such as integrable conditions on \(u\), we show the rigidity for the ambient manifold and classification result of positive solutions. To our knowledge, this is the first rigidity result for equations with gradient terms in the second critical case. Moreover, this result confirms that all solutions must be of the form found in \cite{BV-GH-V2019}.
\keywords{semilinear elliptic equations, Liouville theorem, critical exponent, the invariant tensor technique.}\\
\textbf{2020 Mathematics Subject Classification:} 35J91, 35B08,  35B53.
\end{abstract}

\tableofcontents
\section{Introduction}\label{sec-introduction}
Let \((M^n, g)\) be a complete, connected, and noncompact Riemannian manifold of dimension \(n\geqslant3\) with nonnegative Ricci curvature. 
Consider a class of semilinear equations with a gradient term 
\begin{equation}\label{Meq-gradient}
-\Delta u = f(u)|\nabla u|^q\quad\text{in }M^n,
\end{equation}
where \(\Delta\) is the Laplace-Beltrami operator, \(\nabla\) is the gradient operator, and \(f:\mathbb R_+\to\mathbb R_+\) is continuous.
The equation \eqref{Meq-gradient} has a strong physical background. Such as a qualitative mathematical model studying the groundwater flow in a water-absorbing fissured porous rock in one spatial dimension (see e.g., \cite{BBCP2000} ), the porous media equation (see e.g., \cite{P2018}).

When \(f=0\), equation \eqref{Meq-gradient} becomes \(\Delta u=0\). Cheng-Yau \cite{CY1975} proved that any positive or bounded solution from above or below to \eqref{Meq-gradient} is a constant on Riemannian manifolds with nonnegative Ricci curvature. 

For \(q=0\), the equation \eqref{Meq-gradient} becomes
\begin{equation}\label{Meq-q-0}
-\Delta u =f(u)\quad\text{in }M^n.   
\end{equation}
For more general nonlinearity \(f(x,u)\), with additional assumptions imposed on \(f\), Gidas-Spruck \cite[Theorem 6.1]{GS1981} proved that solutions of equation \eqref{Meq-q-0} must be constant. In the Euclidean case, Serrin-Zou \cite{SZ2002} studied the \(m\)-Laplacian case. By focusing on \(m=2\), they proposed the subcritical condition on \(f\), that is
\[f(u)\geqslant0,\quad (p-1)f(u)-uf'(u)\geqslant0\]
for \(u>0, 1<p<\frac{2n}{n-2}\). Under the subcritical condition on \(f\), they obtained a Liouville theorem of equation \eqref{Meq-q-0} in \(\mathbb R^n\), that is if one of following condition holds:

(i) If \(n=2, 3\) and \(f\) is subcritical,
 
(ii) If \(n\geqslant 4\), \(f(u)\geqslant Cu\) for sufficiently large \(u\), and \(f\) is subcritical,

\noindent then equation \eqref{Meq-q-0} admits only constant solutions. For more Liouville theorems involving equation \eqref{Meq-q-0}, see \cite{L2025, L2025-1, LZ2003, MW2024, Z2025} and the references therein.

A typical case for equation \eqref{Meq-q-0} is
\begin{equation}\label{Meq-q-0-ty}
-\Delta u =u^p\quad\text{in }M^n.
\end{equation}
This equation is the famous Lane-Emden equation, which is related to the Yamabe problem. Gidas-Spruck \cite{GS1981} used differential identities to derive a Liouville theorem, namely, there is no positive solution when \(1<p<\frac{n+2}{n-2}\) on a complete Riemannian manifold of dimension \(n\geqslant3\) with nonnegative Ricci tensor. Grigor’yan-Sun \cite{GS2014} proved that equation \eqref{Meq-q-0-ty} admits no positive solution under condition \(\operatorname{vol}(B_R(x_0))\leqslant CR^{\frac{2p}{p-1}}\ln^{\frac{1}{p-1}}R\), corresponding with the Serrin subcritical case \(p<\frac{n}{n-2}\) in dimension \(n\), when \(p>1\) on a connected geodesically complete Riemannian manifold.

In the critical case  \(p=\frac{n+2}{n-2}\), the equation \eqref{Meq-q-0-ty} becomes
\[-\Delta u =u^{\frac{n+2}{n-2}}\quad\text{in }M^n, \]
which is the Euler-Lagrange equation of the Sobolev inequality. The critical case in \(\mathbb R^n\) has been successfully resolved, as can be seen in \cite{ CGS1989, CL1991,GNN1981}, all positive solutions must be of the form 
\begin{equation}\label{solution-q-0}
u(x)=\Big(\frac{\sqrt{n(n-2)}\lambda}{1+(\lambda|x-x_0|)^2}\Big)^{\frac{n-2}{2}}.
\end{equation}
On generally Riemannian manifolds, the classification result turns to the rigidity result: the manifold must be \(\mathbb R^n\) if there exists a nonconstant solution. However, deriving the rigidity result is more challenging. Fogagnolo-Malchiodi-Mazzieri \cite{FMM2023} got the rigidity result under the assumption \(u=O(r^{-\frac{n-2}{2}})\) at infinity. It is worth mentioning that Catino-Monticelli \cite{CM2022} obtained the rigidity result and classified it with \(n=3\), while \(n\geqslant4\) either under the finite energy condition or under the decay assumption of the solution at infinity. Ciraolo-Farina-Polvara \cite{CFP2024} extended these results to \(n=4,5\) without any finite energy condition by the \(P\)-function method. Catino-Monticelli-Roncoroni \cite{CMR2023}, Ou \cite{O2025}, and Sun-Wang \cite{SW2025} extended the classification results on manifolds to \(m\)-Laplacian case.

Now, we return to the case with the gradient term. A crucial prototype of semilinear equation \eqref{Meq-gradient} is the case \(f(u)=u^p\), i.e.,
\begin{equation}\label{eq-u^p}
-\Delta u = u^p|\nabla u|^q\quad\text{in }M^n.
\end{equation}
There are three important curves of \((p,q)\) for this equation, namely the homogeneous curve \(A_0(p,q)=0\), the first critical curve \(A_1(p,q)=0\), and the second critical curve \(A_2(p,q)=0\), with
\[A_0(p,q)=p+q-1,\quad A_1(p,q)=(n-2)p+(n-1)q-n,\]
\[A_2(p,q)=(n-2)p+(n-1)q-\Big(n+\frac{2-q}{1-q}\Big),\quad0\leqslant q<1.\]
As shown in Figure \ref{fig-1}, we mark these three curves with \(A_0\), \(A_1\), and \(A_2\) respectively. We introduce
\[2_*(q)=\frac{2-q}{n-2}(n-1),\quad 2^*(q)=\frac{2-q}{n-2}(n+\frac{q}{1-q}),\]
thus the first subcritical range \(A_1(p,q)<0\) is equivalent with \(p<2_*(q)-1\), and the second subcritical range \(A_2(p,q)<0\) is equivalent with \(p<2^*(q)-1\) when \(0\leqslant q<1\).

In the special case \(p=0\), equation \eqref{eq-u^p} becomes \(-\Delta u=|\nabla u|^q\), called as the Hamilton-Jacobi equation. For this equation, Lions \cite{L1985} obtained a Liouville theorem for \(q>1\) in \(\mathbb R^n\). Bidaut-V\'{e}ron, Huidobro, and V\'{e}ron \cite{BV-GH-V2014} established the gradient estimate to \(m\)-Laplacian case and obtained some Liouville theorems on complete noncompact manifolds, which satisfy a lower bound estimate on the Ricci curvature and sectional curvature.

In the Euclidean case, there has already been a great deal of work on equation \eqref{eq-u^p}. In the first subcritical case \(p<2_*(q)-1\), any supersolution to \eqref{eq-u^p} must be constant, see \cite{BP-GM-Q2016, BV-GH-V2019, CHZ2022, CM1997, F2009, MP2001}. Those past proofs are usually composed of three cases split by the homogeneous curve \(A_0(p,q)=0\), which is avoided in our proof of Theorem \ref{thm-Serrin}. Besides, for \(q\geqslant2\) and \(p\geqslant0\), Filippucci-Pucci-Souple \cite{FPS2020} obtained that any positive bounded solution to equation \eqref{eq-u^p} is constant. Later, Bidaut-V\'{e}ron \cite{BV2021} extended this result to the \(m\)-Laplace equation for \(q\geqslant m\) and established a Liouville theorem with \(p\geqslant 0\).

Now we concentrate on the second subcritical range \(p<2^*(q)-1\) with \(0\leqslant q<1\) for equation \eqref{eq-u^p} in \(\mathbb R^n\). For \(0<q<2\) and \(p>0\), Bidaut-V\'{e}ron, Huidobro, and V\'{e}ron \cite{BV-GH-V2019} conjectured that any solution to \eqref{eq-u^p} must be constant in the second critical range. They gave a positive answer to the left of the curve \(G(p,q)=0\) in Figure \ref{fig-1}, where 
\[G(p,q)=\big((n-1)^2q+n-2\big)p^2+c(q)p-nq^2,\]
\[c(q)=n(n-1)q^2-(n^2+n-1)q-n-2.\]
Besides, they also showed that the second critical curve is optimal by proving the existence of a nonconstant solution in the second supercritical range \(A_2(p,q)>0\). Using Bernstein's technique, Chang-Hu-Zhang \cite{CHZ2022} derived a Liouville theorem in \(m\)-Laplacian case, and the left of the curve \(C\) in Figure \ref{fig-1} shows their range when \(m=2\). An exciting result was obtained by Ma-Wu \cite{MW2023}. With the help of integral identities, they completely solved the conjecture in \cite{BV-GH-V2019} when \(0<q<\frac{1}{n-1}\), and extended the range \(G(p,q)<0\) in \cite{BV-GH-V2019} to \(H(p,q)<0\) when \(q\geqslant\frac{1}{n-1}\), where
\[H(p,q)=p^2+\Big(\frac{n-1}{n-2}q-\frac{n^2-3}{(n-2)^2}\Big)p+\frac{1-(n-1)q}{(n-2)^2}.\]

\begin{figure}[htbp]\small
\centering
\includegraphics[width=0.8\textwidth]{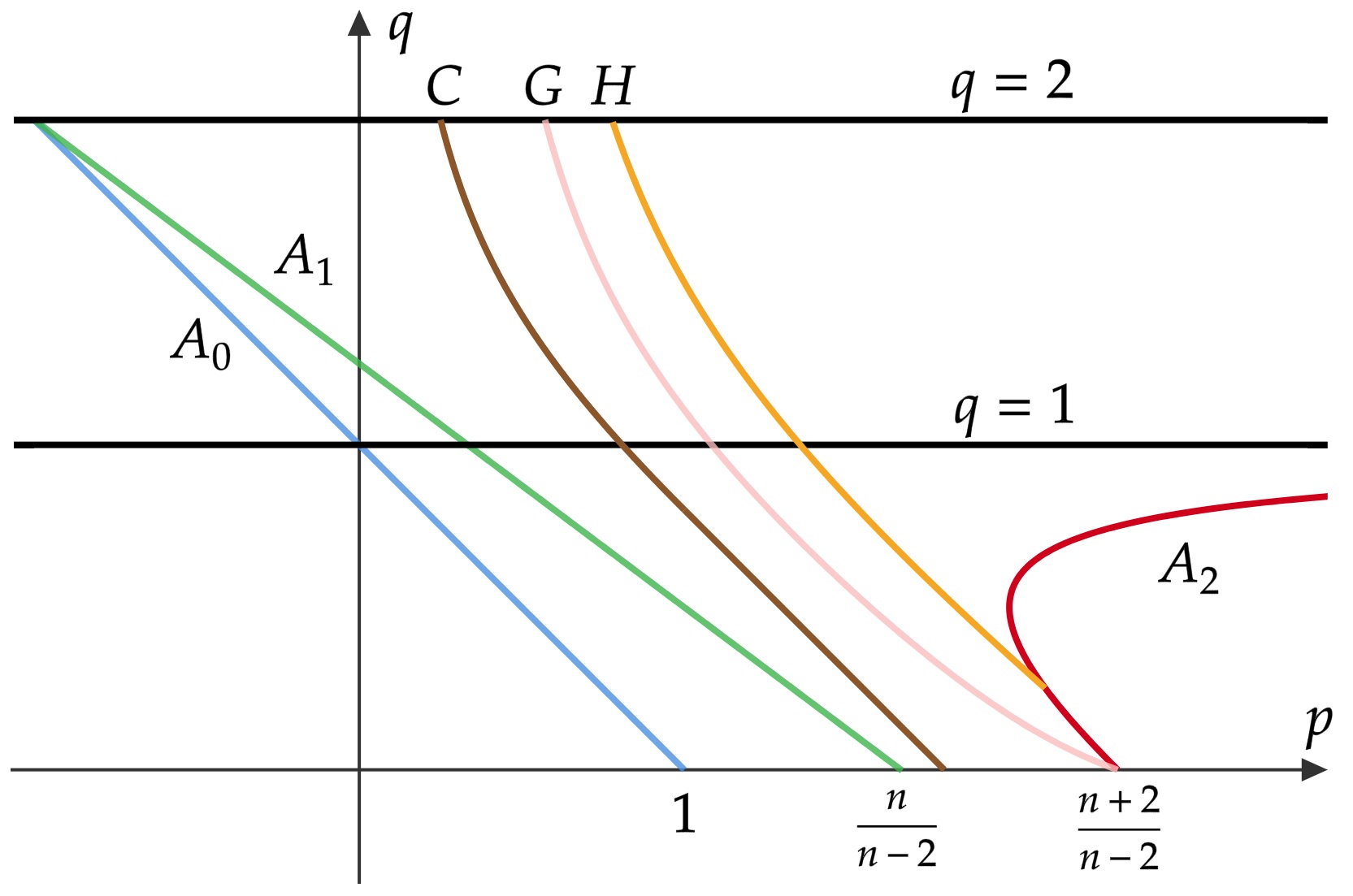}
\caption{some parameter ranges of equation \eqref{eq-u^p} in \(\mathbb R^n\) when \(n=5\)}
\label{fig-1}
\end{figure}

Return to the Riemannian manifold case. Sun-Xiao-Xu \cite{SXX2022} obtained some Liouville theorems of equation \eqref{eq-u^p} under the growth of the geodesical ball and some range \((p,q)\in\mathbb R^2\) on a complete, noncompact Riemannian manifold. He-Hu-Wang \cite{HHW2023} obtained that any \(C^1\) smooth positive solution of equation \eqref{eq-u^p} is a constant when \(p+q<\frac{n+3}{n-1}\) on complete noncompact Riemannian manifold with nonnegative Ricci curvature. More generally, both \cite{HHW2023} and \cite{SXX2022} studied \(m\)-Laplacian case.

Thus, research on Liouville theorems for equation \eqref{eq-u^p} on manifolds is relatively scarce, let alone equation \eqref{Meq-gradient} with general semilinear term \(f(u)\), which motivates us to study equation \eqref{Meq-gradient} on manifolds. First, we consider supersolutions to \eqref{Meq-gradient}.

\begin{definition}
We say that \(u\) is a supersolution to \eqref{Meq-gradient} if \(u\in H_{\mathrm{loc}}^1(M^n)\cap C(M^n)\) satisfies
\begin{align}\label{weak-supsol}
\int_{M^n} \nabla u\cdot\nabla\varphi\geqslant\int_{M^n}f(u)|\nabla u|^q\varphi,\quad\varphi\in C_c^\infty(M^n),~\varphi\geqslant 0.
\end{align}
\end{definition}

\begin{remark}
By an approximation argument, the test function \(\varphi\) may be chosen from \(H^1(M^n)\) with compact support.
\end{remark}

Under an integrable condition on \(f\), we obtain the following Liouville theorem by Serrin's technique.

\begin{theorem}\label{thm-Serrin}
Let \((M^n, g)\) be a complete, connected, and noncompact  Riemannian manifold of dimension \(n\geqslant3\) with nonnegative Ricci curvature. Assume that \(f\in C(\mathbb R_+;\mathbb R_+)\) satisfies
\begin{equation}\label{cond-Serrin}
\int_0^1[f(v)]^{-\frac{1}{2_*(q)-1}}\mathrm dv<+\infty.
\end{equation}
Then any positive supersolution to \eqref{Meq-gradient} must be constant.
\end{theorem}

When \(\liminf_{v\to0^+}f(v)>0\), the condition \eqref{cond-Serrin} trivially holds.

\begin{corollary}
Let \((M^n, g)\) be a complete, connected, and noncompact  Riemannian manifold of dimension \(n\geqslant3\) with nonnegative Ricci curvature. Assume that \(f\in C(\mathbb R_+;\mathbb R_+)\) satisfies
\[\liminf_{v\to0^+}f(v)>0.\]
Then any positive supersolution to \eqref{Meq-gradient} must be constant.
\end{corollary}

Theorem \ref{thm-Serrin} recovers the first subcritical range \(p<2_*(q)-1\) by taking \(f(u)=u^p\).

\begin{corollary}
Let \((M^n, g)\) be a complete, connected, and noncompact  Riemannian manifold of dimension \(n\geqslant3\) with nonnegative Ricci curvature. Then any positive supersolution to \eqref{eq-u^p} must be constant if \(p<2_*(q)-1\),
\end{corollary}

For solutions to \eqref{Meq-gradient}, we concentrate on the cases \(0<q<\frac{1}{n-1}\). We establish the following Liouville theorem. Our proof also works for the case \(q=0\), which recovers the \(m=2\) case in Serrin-Zou \cite{SZ2002} in \(M=\mathbb R^n\).

\begin{theorem}\label{thm-subcritical}
Let \((M^n, g)\) be a complete, connected, and noncompact Riemannian manifold of dimension \(n\geqslant3\) with nonnegative Ricci curvature, and \(0<q<\frac{1}{n-1}\). Assume that \(f\in C^1(\mathbb R_+;\mathbb R_+)\) satisfies the following conditions:
\begin{equation}\label{cond-sub-1}
f(u)\geqslant u^{\beta_0},~u>N,\quad\text{for some }\beta_0\in(1-q,2^*(q)-1)\text{ and }N>0,
\end{equation}
\begin{equation}
uf'(u)\leqslant\alpha_0 f(u),\quad\text{for some }\alpha_0\in[\beta_0,2^*(q)-1).
\end{equation}
Then any positive \(C^3\) solution to \eqref{Meq-gradient} must be constant if \(p<2^*(q)-1\).
\end{theorem}

Theorem \ref{thm-subcritical} recovers the second subcritical range \(p<2^*(q)-1\) by taking \(f(u)=u^p\).

\begin{corollary}
Let \((M^n, g)\) be a complete, connected, and noncompact Riemannian manifold of dimension \(n\geqslant3\) with nonnegative Ricci curvature, and \(0<q<\frac{1}{n-1}\). Then any positive \(C^3\) solution to \eqref{eq-u^p} must be constant.
\end{corollary}

By taking \(M=\mathbb R^n\), this corollary recovers the case \(0<q<\frac{1}{n-1}\) in \cite[Theorem 1.3]{MW2023}, where Ma and Wu solved the conjectured raised in \cite{BV-GH-V2019}.

Lastly, we consider the second critical case \(f(u) = u^{2^*(q)-1}\). Bidaut-V\'{e}ron, Huidobro, and V\'{e}ron \cite[Theorem D]{BV-GH-V2019} found an explicit 1-parameter family of positive radial solutions in \(\mathbb R^n\)
\[u_c(r)=c(Kc^{\frac{(2-q)^2}{(n-2)(1-q)}}+r^{\frac{2-q}{1-q}})^{-\frac{(n-2)(1-q)}{2-q}},\quad c>0,~K=K(n,q)>0.\]
By inserting them into equation \eqref{Meq-gradient} with \(f(u) = u^{2^*(q)-1}\), we rewrite the solutions as
\begin{equation}\label{solution}
U_{\lambda,x_0}(x):=\Big(\frac{(n+\frac{q}{1-q})^{\frac{1}{2-q}}(n-2)^{\frac{1-q}{2-q}}\lambda}{1+(\lambda|x-x_0|)^{\frac{2-q}{1-q}}}\Big)^{\frac{(n-2)(1-q)}{2-q}},\quad\lambda>0,~x_0\in\mathbb R^n.
\end{equation}
For \(q=0\), the solution \eqref{solution} is equivalent to \eqref{solution-q-0}. The rigidity of manifolds has been rigorously established in low-dimensional settings for \(q=0\), but for \(0<q<1\), the classification of equation \eqref{Meq-gradient} in the second critical case \(f(u) = u^{2^*(q)-1}\) on \(\mathbb R^n\) is open, let alone the rigidity of manifolds.

With the help of the invariant tensor \(\operatorname{\mathbf E}\) and its corresponding pseudo-invariant function \(l\) (i.e., \(\nabla l\) is only composed of terms involving the invariant tensor \(\operatorname{\mathbf E}\)), we obtain the rigidity for the ambient manifold and classify positive solutions of the second critical equation, \eqref{Meq-gradient} with \(f(u)=u^{2^*(q)-1}\), without any assumption in a low-dimensional setting. To our knowledge, this is the first rigidity result in the second critical case of equations with gradient terms.

\begin{theorem}\label{thm-critical}
Let \((M^n, g)\) be a complete, connected, and noncompact Riemannian manifold of dimension \(n\in\{3,4,5\}\) with nonnegative Ricci curvature, \(0<q<\frac{1}{n-1}\) and \(f(u)=u^{2^*(q)-1}\) in \eqref{Meq-gradient}. Let \(u\in C^3(M)\) be a positive solution to \eqref{Meq-gradient}, then either \(u\) is constant, or \((M^n, g)\) is isometric to \(\mathbb R^n\) and \(u\) is the form given by \eqref{solution}.
\end{theorem}

Differential identities are crucial in studying properties of an elliptic equation. However, establishing differential identities usually relies on some geometric background, such as the Bochner formula. Thus, finding a proper differential identity may be challenging for a sophisticated equation. In \cite{MW2024}, Ma-Wu (the third author of this paper) developed the invariant tensor technique, a powerful method of finding vector fields and differential identities without any geometric instructions, to study semilinear equations \(-\Delta u=f(u)\) on \(M^n\), which gave a new proof of Bidaut-V\'{e}ron and V\'{e}ron's rigidity result in \cite{BV-V1991}. Using the new technique, Ma-Ou-Wu \cite{MOW2025} reconstructed Jerison-Lee's identity to answer the question raised by Jerison-Lee \cite{JL1988}. Besides, Ma-Ou-Wu \cite{MOW2023} verified the subcritical rigidity conjecture raised by Wang \cite{W2022} via this new technique with sophisticated computations. Recently, Guo-Zhang \cite{GZ2025} utilised the invariant tensor technique to give a new proof of semilinear equations again. The invariant tensor technique is also widely applied to classify quasilinear cases or involving gradient terms, see \cite{LWY2025, Z2024}. 


Although we concentrate only on the case \(0<q<\frac{1}{n-1}\), we make the first attempt to successfully extend the invariant tensor technique to equations with gradient terms. The pseudo-invariant function \(l\), which plays a crucial role in the classification process, is found by differential invariance. However, it becomes difficult to find \(l\) through variable transformations as previously done. The authors have extended the invariant tensor technique to the equation \eqref{Meq-gradient} to \(m\)-Laplace in \cite{WZ}, where a wider range of \(q\) is studied with more sophisticated computations.

Throughout this paper, we assume constant \(R>1\), \(\varepsilon>0\) small enough, and \(C>0\) are constants independent of \(R\). Let \(B_R\) be the geodesic ball centered at some fixed point with radius \(R\). Unless otherwise specified, we employ the summation convention for repeated indices from \(1\) to \(n\). In proofs involving manifolds, we choose a local frame, \(\varphi_i\) denotes the covariant derivative of the function \(\varphi\). Let \(\eta\) be a smooth cutoff function supported in \(B_{2R}\) satisfying \(\eta\equiv 1\) in \(B_R\), \(0\leqslant\eta\leqslant 1\), and \(|\nabla\eta|\leqslant CR^{-1}\).

This paper is structured as follows. In Section 2, we prove Theorem \ref{thm-Serrin} by using Serrin's technique with a delicate test function \(b(u)\eta^n\) and a three-term Young's inequality. In Section 3, using the invariant tensor technique, we establish a key differential identity and necessary propositions to prove Theorem \ref{thm-subcritical}. With the help of this differential identity, the proof of Theorem \ref{thm-critical}, classification of solutions in the critical case, is finished in Section 4.
\section{The first subcritical case}
 
In this section, we devote to showing the Liouville theorem for equation \eqref{Meq-gradient} in the first subcritical case, that is, we prove Theorem \ref{thm-Serrin}.

\textbf{Proof of Theorem \ref{thm-Serrin}.}
Let \(u\) be a positive supersolution of \eqref{Meq-gradient}, and take \(\varphi=b(u)\eta^n\) in \eqref{weak-supsol}, where \(b\in C^1(\mathbb R_+)\) satisfies \(b>0\) and \(b'<0\). By the definition of \eqref{weak-supsol}, it holds
\begin{equation}\label{ineq-Serrin-1}
\int_{M^n} b(u)f(u)|\nabla u|^q\eta^n-\int_{M^n}b'(u)|\nabla u|^2\eta^n\leqslant n\int_{B_{2R}\backslash B_R}b(u)\eta^{n-1}\nabla u\cdot\nabla\eta.
\end{equation}
Using Young's inequality, we have
\begin{align}\label{ineq-Serrin-2}
n\int_{B_{2R}\backslash B_R}b(u)\eta^{n-1}\nabla u\cdot\nabla\eta\leqslant~&A\int_{B_{2R}\backslash B_R} b(u)f(u)|\nabla u|^q\eta^n-A\int_{B_{2R}\backslash B_R}b'(u)|\nabla u|^2\eta^n\nonumber\\
~&+CA^{1-n}\int_{B_{2R}\backslash B_R}[f(u)]^{-\frac{n-2}{2-q}}[\tilde b'(u)]^{-\frac{n-(n-1)q}{2-q}}|\nabla\eta|^n,
\end{align}    
where \(A>0\), \(\tilde b(u)=[b(u)]^{-\frac{2-q}{n-(n-1)q}}\). Choosing \(A=\frac 1 2\), substitution of \eqref{ineq-Serrin-2} into \eqref{ineq-Serrin-1} gives
\[\int_{B_R} b(u)f(u)|\nabla u|^q-\int_{B_R}b'(u)|\nabla u|^2\leqslant C\int_{B_{2R}\backslash B_R}[f(u)]^{-\frac{n-2}{2-q}}[\tilde b'(u)]^{-\frac{n-(n-1)q}{2-q}}|\nabla\eta|^n.\]

Take \(\tilde b(u)=\int_0^u[f(v)]^{-\frac{1}{2_*(q)-1}}\mathrm dv\). Using the Bishop-Gromov volume comparison theorem \cite[Lemma 7.1.4]{P2006}, we have
\[\int_{B_R} b(u)f(u)|\nabla u|^q-\int_{B_R}b'(u)|\nabla u|^2\leqslant C,\]
we arrive at \(\lim_{R\to\infty}\big(\int_{B_{2R}\backslash B_R} b(u)f(u)|\nabla u|^q-\int_{B_{2R}\backslash B_R}b'(u)|\nabla u|^2\big)=0\).
    
By the above estimates, for any \(A > 0\), substituting \eqref{ineq-Serrin-2} into \eqref{ineq-Serrin-1} and letting \(R \to \infty\) yields
\[\int_{M^n} b(u)f(u)|\nabla u|^q\eta^n-\int_{M^n}b'(u)|\nabla u|^2\eta^n\leqslant CA^{1-n}.\]
The proof of Theorem \ref{thm-Serrin} is completed by letting \(A\to\infty\).
\qed
\section{The second subcritical case}\label{sec-subcritical}

In this section, we discuss the second subcritical case and prove Theorem \ref{thm-subcritical}. We always assume that \(u\in C^3(M^n)\) is a positive solution to equation \eqref{Meq-gradient} in this section. Set 
\[Z=\{x\in M^n\,:\,|\nabla u|=0\}.\] All subsequent discussion in this section focuses on \(Z^c:=M^n\backslash Z\).

Motivated by the invariant tensors technique introduced in \cite{MW2024}, define
\[\mathbf E=\nabla^2 u-\frac{\nabla u\otimes\nabla u}{a(u)}+c\frac{\Delta u}{|\nabla u|^2}\nabla u\otimes\nabla u-\frac{1}{n}\big((1+c)\Delta u-\frac{|\nabla u|^2}{a(u)}\big)g,\]
as the invariant tensor with its components \(E_{ij}:=\mathbf E(\partial_i,\partial_j)\), that is
\begin{equation}\label{invariant}
    E_{ij}=u_{ij}-\frac{u_i u_j}{a(u)}+c\frac{\Delta u}{|\nabla u|^2}u_i u_j-\frac{1}{n}\big((1+c)\Delta u-\frac{|\nabla u|^2}{a(u)}\big)g_{ij},
\end{equation}
where \(a:\mathbb{R}_+\to \mathbb{R}_+\) is a continuously differentiable function, and \(c\in\mathbb R\). Clearly, \(\mathbf E\) is symmetric, i.e., \(E_{ij}=E_{ji}\). The following lemma establishes the positivity of \(E_{ij}\).

\begin{lemma}\label{lem-positivity}
Let \(x\in Z^c\), it holds
\[\frac{|\mathbf E\cdot\nabla u|^2}{|\nabla u|^2}\leqslant\frac{n-1}{n}|\mathbf E|^2.\]
\end{lemma}

\begin{proof}
Without loss of generality, we assume \(u_1=|\nabla u|>0\), \(u_\iota=0\), \(u_{\iota\xi}=0\), \(\forall2\leqslant\iota, \xi\leqslant n\), \(\iota\neq\xi\), then \(E_{\iota\xi}=0\), \(E_{1\iota}=u_{1\iota}\). By the trace-free property of \(E_{ij}\), \(E_{11}=-\sum_{\iota=2}^n E_{\iota\iota}\). Therefore,
\begin{align*}
\frac{n-1}{n}E_{ij}E^{ij}-\frac{1}{|\nabla u|^2}E_{ij}u^jE^{ik}u_k=~&\frac{n-2}{n}\sum_{\iota=2}^n|E_{1\iota}|^2+\frac{n-1}{n}\sum_{\iota=2}^n|E_{\iota\iota}|^2-\frac 1 n|E_{11}|^2\\
=~&\frac{n-2}{n}\sum_{\iota=2}^n(|E_{1\iota}|^2+|E_{\iota\iota}|^2)-\frac 2 n\sum_{2\leqslant\iota<\xi\leqslant n}E_{\iota\iota}E_{\xi\xi}\\
=~&\frac{n-2}{n}\sum_{\iota=2}^n|E_{1\iota}|^2+\frac 1 n\big(\sum_{2\leqslant\iota<\xi\leqslant n}|E_{\iota\iota}-E_{\xi\xi}|^2\big)
\geqslant~0.
\end{align*}
The lemma is proved.
\end{proof}
We would mention that if \(\operatorname{\mathbf A}\) and \(\operatorname{\mathbf B}\) are covariant tenors, then 
\begin{equation}\label{AB-cov}
\operatorname{\mathbf A}:\operatorname{\mathbf B}=g^{ij}\operatorname{\mathbf A}_{jk}g^{kl}\operatorname{\mathbf B}_{li}.
\end{equation}
Now, we show a key differential identity.

\begin{proposition}
Let \(\beta, d\in\mathbb R\). Then in \(Z^c\),
\begin{align}\label{id}
&u^{-\beta}|\nabla u|^{-d}\operatorname{div}(u^{\beta}|\nabla u|^d\mathbf E\cdot\nabla u)\nonumber\\
=~&\operatorname{tr}\mathbf E^2+\frac{d}{|\nabla u|^2}\mathbf E^2:\nabla u\otimes\nabla u+\big[\frac{\beta}{u}+\frac{2+(n-1)d}{n}\frac{1}{a(u)}\big]\mathbf E:\nabla u\otimes\nabla u\nonumber\\
&+\big\{\frac{(n-1)(1+c)q}{n}-2c-\frac{(n-1)c-1}{n}d\big\}\frac{\Delta u}{|\nabla u|^2}\mathbf E:\nabla u\otimes\nabla u+\operatorname{Ric}(\nabla u,\nabla u)\nonumber\\
&+\frac{n-1}{n}\big(a'(u)-\frac{n-2}{n}\big)\frac{|\nabla u|^4}{a^2(u)}+\frac{n-1}{n^2}(1+c)[q+(n-(n-1)q)c](\Delta u)^2\nonumber\\
&-\frac{n-1}{n}\big\{(1+c)f'(u)-\big[1+\frac{2-(n-1)q}{n}(1+c)\big]\frac{f(u)}{a(u)}\big\}|\nabla u|^{2+q}.
\end{align}
\end{proposition}

\begin{proof}
From the definition of \(E_{ij}\), we have
\[E_{ij}u^i=u_{ij}u^i-\frac{n-1}{n}\frac{|\nabla u|^2 u_j}{a(u)}+\frac{(n-1)c-1}{n}\Delta u u_j,\]
\begin{align*}
E_{ij},^i=~&\frac{n-1-c}{n}(\Delta u)_j+R_{ij}u^i+\frac{n-1}{n}\frac{a'(u)}{a^2(u)}|\nabla u|^2 u_j-\frac{n-2}{n}\frac{u_{ij}u^i}{a(u)}\\
&-\frac{\Delta u u_j}{a(u)}+c\frac{(\Delta u)^i}{|\nabla u|^2}u_i u_j+c\frac{(\Delta u)^2}{|\nabla u|^2}u_j+c\frac{\Delta uu_{ij}u^i}{|\nabla u|^2}-2c\frac{\Delta u u^{ik}u_i u_k u_j}{|\nabla u|^4}.
\end{align*}
Differentiating \eqref{Meq-gradient} yields 
\[(\Delta u)_i=\frac{f'(u)}{f(u)}\Delta u u_i +q\frac{\Delta u}{|\nabla u|^2}u_{ij} u^j.\] 
Replacing \(u_{ij}\) with \(E_{ij}\) gives
\begin{align*}
E_{ij},^i=~&-\frac{n-2}{n}\frac{E_{ij}u^i}{a(u)}+\big(\frac{n-1}{n}q+\frac{n-q}{n}c\big)\frac{\Delta u}{|\nabla u|^2}E_{ij}u^i-(2-q)c\frac{\Delta u E_{ik}u^i u^k u_j}{|\nabla u|^4}\\
&+\frac{n-1}{n}\big(a'(u)-\frac{n-2}{n}\big)\frac{|\nabla u|^2}{a^2(u)}u_j+\frac{n-1}{n^2}(1+c)[q+(n-(n-1)q)c]\frac{(\Delta u)^2}{|\nabla u|^2} u_j\\
&+\frac{n-1}{n}\big\{(1+c)\frac{f'(u)}{f(u)}-\big[1+\frac{2-(n-1)q}{n}(1+c)\big]\frac{1}{a(u)}\big\}\Delta u u_j+R_{ij}u^i.
\end{align*}
Then
\begin{align}\label{Eu-id}
(E_{ij}u^j),^i=~&E_{ij}E^{ij}+\frac{2}{n}\frac{E_{ij}u^iu^j}{a(u)}+\big(\frac{(n-1)(1+c)q}{n}-2c\big)\frac{\Delta uE_{ij}u^iu^j}{|\nabla u|^2}+R_{ij}u^iu^j\nonumber\\
&+\frac{n-1}{n}\big(a'(u)-\frac{n-2}{n}\big)\frac{|\nabla u|^4}{a^2(u)}+\frac{n-1}{n^2}(1+c)[q+(n-(n-1)q)c](\Delta u)^2\nonumber\\
&+\frac{n-1}{n}\big\{(1+c)\frac{f'(u)}{f(u)}-\big[1+\frac{2-(n-1)q}{n}(1+c)\big]\frac{1}{a(u)}\big\}\Delta u |\nabla u|^2.
\end{align}
A direct computation has
\begin{align}\label{Eu-id-1}
&u^{-\beta}|\nabla u|^{-d}(u^{\beta}|\nabla u|^dE_{ij} u^j),^i\nonumber\\
=~&\frac{\beta}{u}E_{ij}u^iu^j+(E_{ij}u^j),^i+\frac{d}{|\nabla u|^2}E_{ij}u^j\big(E^{ki}u_k+\frac{n-1}{n}\frac{|\nabla u|^2u^i}{a(u)}-\frac{(n-1)c-1}{n}\Delta uu^i\big).
\end{align}
Combining \eqref{Meq-gradient} and \eqref{Eu-id} with \eqref{Eu-id-1} and using \eqref{AB-cov}, a direct computation yields \eqref{id}.
\end{proof}

\begin{proposition}\label{prop-id}
Let \(a(u)=\big(\frac{n-2}{n}+\varepsilon\big)u\), \(\beta=-\frac{2+(n-1)d}{n-2+n\varepsilon}\), \(c=-\frac{q-\varepsilon}{n-(n-1)q}\), \(d=\frac{(n-1)(1+c)q-2nc}{(n-1)c-1}\) for sufficiently small \(\varepsilon>0\). Under the assumptions of Theorem \ref{thm-subcritical}, there exists \(\delta>0\) depending on \(n,q,\alpha_0,\varepsilon\), such that in \(Z^c\),
\begin{equation}\label{ineq-id}
\operatorname{div}(u^{\beta}|\nabla u|^d\mathbf E\cdot\nabla u)\geqslant \delta u^{\beta}|\nabla u|^d\Big(\frac{|\mathbf E\cdot\nabla u|^2}{|\nabla u|^2}+\frac{|\nabla u|^4}{u^2}+(\Delta u)^2\Big).  
\end{equation}
\end{proposition}

\begin{remark}
Selections in this proposition ensure that coefficients of all \(\mathbf E:\nabla u\otimes\nabla u\) terms in \eqref{id} vanish, and a little \(\frac{|\nabla u|^4}{a^2(u)}\) and \((\Delta u)^2\) are retained.
\end{remark}

\begin{proof}
According to Lemma \ref{lem-positivity} and \(\operatorname{Ric}\geqslant0\), identity \eqref{id} yields
\begin{align*}
&u^{-\beta}|\nabla u|^{-d}(u^{\beta}|\nabla u|^dE_{ij} u^j),^i\\
\geqslant~&\big(\frac{n}{n-1}+d\big)\frac{E_{ij}u^jE^{ik}u_k}{|\nabla u|^2}+\frac{n(n-1)\varepsilon}{(n-2+n\varepsilon)^2}\frac{|\nabla u|^4}{u^2}+\frac{n-1}{n^2}\frac{n(1-q)+\varepsilon}{n-(n-1)q}\varepsilon(\Delta u)^2\\
&-\frac{n-1}{n}\frac{n(1-q)+\varepsilon}{n-(n-1)q}\big[f'(u)-\big(\frac{1}{1+c}+\frac{2-(n-1)q}{n}\big)\frac{f(u)}{a(u)}\big]|\nabla u|^{2+q}.
\end{align*}
Note that 
\[\lim_{\varepsilon\to0^+}\big(\frac{n}{n-1}+d\big)=[n-(n-1)q](\frac{1}{n-1}-q)>0.\]
By the subcritical condition of Theorem \ref{thm-subcritical}, it yields
\begin{align*}
&\lim_{\varepsilon\to 0^+}\big[f'(u)-\big(\frac{1}{1+c}+\frac{2-(n-1)q}{n}\big)\frac{f(u)}{a(u)}\big]\\
=~&f'(u)-[2^*(q)-1]\frac{f(u)}{u}\leqslant[\alpha_0-2^*(q)+1]\frac{f(u)}{u}<0.
\end{align*}
This completes the proof.
\end{proof}

\textbf{Proof of Theorem \ref{thm-subcritical}.}
Testing inequality \eqref{ineq-id} with \(\eta^\gamma\), where \(\gamma>0\) is a sufficiently large constant independent of \(R\), yields
\begin{align}\label{ineq-sub-0}
&\int_{Z^c}u^{\beta}|\nabla u|^d\Big(\frac{|\mathbf E\cdot\nabla u|^2}{|\nabla u|^2}+\frac{|\nabla u|^4}{u^2}+(\Delta u)^2\Big)\eta^\gamma\nonumber\\
\leqslant ~& C\int_{\partial Z\cap B_{2R}} u^{\beta} |\nabla u|^d |\mathbf E:\nabla u\otimes\nu|\eta^\gamma +C\int_{Z^c\cap B_{2R}}u^{\beta}|\nabla u|^d|\mathbf E:\nabla u\otimes\nabla\eta|\eta^{\gamma-1},  
\end{align}
where \(\nu\) is the unit outer normal vector field on \(\partial Z\).

Consider the first term of the RHS of inequality \eqref{ineq-sub-0}. By the definition of \(E_{ij}\), we have
\begin{align}\label{ineq-sub-1}
&\int_{\partial Z\cap B_{2R}}u^{\beta}|\nabla u|^d|\mathbf E:\nabla u\otimes\nu|\eta^\gamma\nonumber\\
\leqslant~&C\int_{\partial Z\cap B_{2R}}u^{\beta}|\nabla u|^d\big(|\nabla^2 u:\nabla u\otimes\nu|+\frac{|\nabla u|^3}{a(u)}+f(u)|\nabla u|^{q+1}\big)\eta^{\gamma}.
\end{align}
From the selection of \(d\) in Proposition \ref{prop-id}, it holds
\[\lim_{\varepsilon\to0^+}d+q+1=1-q[n-(n-1)q]>0,\]
\[\lim_{\varepsilon\to0^+}d-q+2=\lim_{\varepsilon\to0^+}d+q+1+(1-2q)\geqslant\lim_{\varepsilon\to0^+}d+q+1+\frac{n-3}{n-1}>0.\]
Therefore, on \(Z\),
\[|\nabla u|^{d+q+1}=|\nabla u|^{d+3}=|\nabla u|^{d-q+2}=0.\] 
Since \(\nabla\frac{\Delta u}{f(u)}\) is bounded on \(B_{2R}\) due to \(u\in C^3(M^n)\), we have
\[|\nabla u|^d|\nabla^2 u:\nabla u\otimes\nu|\leqslant\frac 1 q|\nabla u|^{d-q+2}\big|\nabla|\nabla u|^q\big|=\frac 1 q|\nabla u|^{d-q+2}\big|\nabla\frac{\Delta u}{f(u)}\big|=0\]
on \(Z\cap B_{2R}\). On the other hand, \(\partial Z\subset Z\) by the closedness of \(Z\), thus
\begin{align}\label{ineq-sub-2}
\int_{\partial Z\cap B_{2R}}u^{\beta}|\nabla u|^d|\mathbf E:\nabla u\otimes\nu|\eta^\gamma\leqslant0.
\end{align}

Now, we focus on the second term of the RHS of inequality \eqref{ineq-sub-0}. Let \(\mathcal X_1:=\{u<N\}\) and \(\mathcal X_2:=\{u\geqslant N\}\), then \(f(u)\geqslant Cu^{\alpha_0}\) on \(\mathcal X_1\), and \(f(u)\geqslant u^{\beta_0}\) on \(\mathcal X_2\). By Young's inequality, it holds
\begin{align}\label{ineq-sub-Young-1}
&\int_{Z^c\cap B_{2R}\cap\mathcal X_1}u^{\beta}|\nabla u|^d|\mathbf E:\nabla u\otimes\nabla\eta|\eta^{\gamma-1}\nonumber\\
\leqslant~&\varepsilon\int_{Z^c}u^{\beta}|\nabla u|^d\Big(\frac{|\mathbf E\cdot\nabla u|^2}{|\nabla u|^2}+\frac{|\nabla u|^4}{u^2}+(\Delta u)^2\Big)\eta^\gamma\nonumber\\
&+C\int_{Z^c\cap\mathcal X_1} u^{\frac{2a_3}{a_4}+\beta}[f(u)]^{-\frac{2a_2}{a_4}}\eta^{\gamma-\frac{1}{a_4}}|\nabla\eta|^{\frac{1}{a_4}}, 
\end{align}
where \(a_2=\frac{1-(d+4)a_4}{2(2-q)}\) and \(a_3=\frac{(d+2q)a_4-q+1}{2(2-q)}\) satisfy the condition
\begin{equation}\label{cond-sub-Young}
a_2,a_3,a_4>0,~a_4<\frac 1 n.
\end{equation}
Noting that
\begin{align*}
&\lim_{\varepsilon\to0^+}d=-q[n+1-(n-1)q],\quad
\lim_{\varepsilon\to0^+}d+4>0,\\
&\lim_{\varepsilon\to0^+}d+2q=-q(1-q)(n-1)<0 ,\\
&\lim_{\varepsilon\to0^+}\big(\frac{1}{d+4}+\frac{1-q}{d+2q}\big)=\lim_{\varepsilon \to 0^+}\frac{(2-q)(1-q)[2-(n-1)q]}{(d+4)(d+2q)}<0.
\end{align*}
Then \eqref{cond-sub-Young} is equivalent to \(0<a_4<\min\big\{\frac 1 n,\frac{1}{d+4}\big\}\) when \(\varepsilon>0\) is small enough. Thus,
\begin{equation}\label{cond-sub}
0<a_4<\begin{cases}
\frac 1 n,&n\geqslant4,\\
\min\big\{\frac 1 3,\frac{1}{2q^2-4q+4}\big\},&n=3,
\end{cases}
\end{equation}
leads to the conclusion that \eqref{cond-sub-Young} holds. 
    
On \(\mathcal X_1\), it holds
\[u^{\frac{2a_3}{a_4}+\beta}[f(u)]^{-\frac{2a_2}{a_4}}\leqslant Cu^{\frac{2a_3}{a_4}-\frac{2+(n-1)d}{n-2+n\varepsilon}-\frac{2a_2}{a_4}\alpha_0},\] 
and 
\begin{align*}
\frac{2a_3}{a_4}-\frac{2+(n-1)d}{n-2+n\varepsilon}-\frac{2a_2}{a_4}\alpha_0>~&\frac{2a_3}{a_4}-\frac{2+(n-1)d}{n-2+n\varepsilon}-\frac{2a_2}{a_4}[2^*(q)-1]\\
\to~&\frac{(2-q)[(n+1-(n-1)q)a_4-1]}{(n-2)(1-q)a_4}\ (\varepsilon\to0^+).
\end{align*}
When \(n\geqslant3\), 
\[\lim_{a_4\to(\frac 1 n)^-}[n+1-(n-1)q]a_4-1=\frac{1-(n-1)q}{n}>0.\]
When \(n=3\), \[\lim_{a_4\to(\frac{1}{2q^2-4q+4})^-}(4-2q)a_4-1=\frac{q(1-q)}{q^2-2q+2}>0.\]
Choosing \(a_4\) sufficiently close to the upper bound in \eqref{cond-sub}, we have
\[\frac{2a_3}{a_4}-\frac{2+(n-1)d}{n-2+n\varepsilon}-\frac{2a_2}{a_4}\alpha_0>0.\]
For a fixed \(\gamma>\frac{1}{a_4}\), by the Bishop-Gromov volume comparison theorem, we obtain \[\int_{Z^c\cap\mathcal X_1} u^{\frac{2a_3}{a_4}+\beta}[f(u)]^{-\frac{2a_2}{a_4}}\eta^{\gamma-\frac{1}{a_4}}|\nabla\eta|^{\frac{1}{a_4}}\leqslant CR^{n-\frac{1}{a_4}}.\]
Substituting it into \eqref{ineq-sub-Young-1}, one has
\begin{align}\label{ineq-sub-Young-1-1}
&\int_{Z^c\cap B_{2R}\cap\mathcal X_1}u^{\beta}|\nabla u|^d|\mathbf E:\nabla u\otimes\nabla\eta|\eta^{\gamma-1}\nonumber\\
\leqslant~&\varepsilon\int_{Z^c}u^{\beta}|\nabla u|^d\Big(\frac{|\mathbf E\cdot\nabla u|^2}{|\nabla u|^2}+\frac{|\nabla u|^4}{u^2}+(\Delta u)^2\Big)\eta^\gamma+CR^{n-\frac{1}{a_4}}.
\end{align}

Similarly, for \(b_2=\frac{1-(d+4)b_4}{2(2-q)}\), \(b_3=\frac{(d+2q)b_4-q+1}{2(2-q)}\) and \(b_4\) satisfying \eqref{cond-sub}, we have
\begin{align}\label{ineq-sub-Young-2}
&\int_{Z^c\cap B_{2R}\cap\mathcal X_2}u^{\beta}|\nabla u|^d|\mathbf E:\nabla u\otimes\nabla\eta|\eta^{\gamma-1}\nonumber\\
\leqslant~&\varepsilon \int_{Z^c}u^{\beta}|\nabla u|^d\Big(\frac{|\mathbf E\cdot\nabla u|^2}{|\nabla u|^2}+\frac{|\nabla u|^4}{u^2}+(\Delta u)^2\Big)\eta^\gamma\nonumber\\
&+C\int_{Z^c\cap\mathcal X_2}u^{\frac{2b_3}{b_4}+\beta}[f(u)]^{-\frac{2b_2}{b_4}}\eta^{\gamma-\frac{1}{b_4}}|\nabla\eta|^{\frac{1}{b_4}}.
\end{align}
On \(\mathcal X_2\), 
\[u^{\frac{2b_3}{b_4}+\beta}[f(u)]^{-\frac{2b_2}{b_4}}\leqslant Cu^{\frac{2b_3}{b_4}-\frac{2+(n-1)d}{n-2+n\varepsilon}-\frac{2b_2}{b_4}\beta_0}.\]
Since \(\forall\beta_0>1-q\), it holds
\begin{align*}
\frac{2b_3}{b_4}-\frac{2+(n-1)d}{n-2+n\varepsilon}-\frac{2b_2}{b_4}\beta_0=\frac{1-q-\beta_0}{2-q}\frac{1}{b_4}+O(1)\ (b_4\to0^+).
\end{align*}
Letting \(b_4\) close to \(0\), one has
\[\frac{2b_3}{b_4}-\frac{2+(n-1)d}{n-2+n\varepsilon}-\frac{2b_2}{b_4}\beta_0<0.\]
Taking \(\gamma=\frac{1}{b_4}>\frac{1}{a_4}\) and \eqref{ineq-sub-Young-2}, according to the Bishop-Gromov volume comparison theorem, we obtain
\begin{align}\label{ineq-sub-Young-2-1}
&\int_{Z^c\cap B_{2R}\cap\mathcal X_2}u^{\beta}|\nabla u|^d|\mathbf E:\nabla u\otimes\nabla\eta|\eta^{\gamma-1}\nonumber\\
\leqslant~&\varepsilon\int_{Z^c}u^{\beta}|\nabla u|^d\Big(\frac{|\mathbf E\cdot\nabla u|^2}{|\nabla u|^2}+\frac{|\nabla u|^4}{u^2}+(\Delta u)^2\Big)\eta^\gamma+CR^{n-\frac{1}{b_4}}.
\end{align}
Hence, when \(\gamma=\frac{1}{b_4}\), by substituting \eqref{ineq-sub-2}, \eqref{ineq-sub-Young-1-1} and \eqref{ineq-sub-Young-2-1} into \eqref{ineq-sub-0}, it yields
\[\int_{Z^c\cap B_R}u^{\beta-2}|\nabla u|^{d+4}\leqslant CR^{n-\frac{1}{b_4}}\to0~(R\to\infty),\]
which means that \(Z^c\) is null, i.e., \(|\nabla u|\equiv0\). Therefore, \(u\) must be constant. 
\qed
\section{The second critical case}\label{sec-critical}

In this section, we focus on the second critical case and prove Theorem \ref{thm-critical}.

We assume that \(u\) is nontrivial. Consider \(f(u)=u^{2^*(q)-1}\), and all notation follows from Section \ref{sec-subcritical}. Argument as in Section \ref{sec-subcritical}, the invariant tensor reduces to
\[E_{ij}=u_{ij}-\frac{n}{n-2}\frac{u_i u_j}{u}-\frac{q}{n-(n-1)q}\frac{\Delta u}{|\nabla u|^2}u_i u_j-\big(\frac{1-q}{n-(n-1)q}\Delta u -\frac{1}{n-2}\frac{|\nabla u|^2}{u}\big)g_{ij}.\]
In this time,
\begin{align}\label{MEiji}
E_{ij},^i=-\frac{E_{ij}u^i}{u}+\frac{q(n-2)(1-q)}{n-(n-1)q}\frac{\Delta u}{|\nabla u|^2}E_{ij}u^i+\frac{(2-q)q}{n-(n-1)q}\frac{\Delta u E_{ik}u^i u^k u_j}{|\nabla u|^4}+R_{ij}u^i.
\end{align}
Define the pseudo-invariant function as
\[l=u^{-2_*(q)}|\nabla u|^{2-q}+\frac{2-q}{2^*(q)}u^{\frac{2-q}{(n-2)(1-q)}},\]
which satisfies \(l_i=(2-q)u^{-2_*(q)}|\nabla u|^{-q} E_{ij}u^j\), meaning that \(l_i\) is a multiple of the invariant tensor.

In the second critical case, applying the selection from Proposition \ref{prop-id} with \(\varepsilon=0\) to identity \eqref{id}, and inserting powers of \(l\), we obtain the following proposition.
\begin{proposition}
Let \(0\leqslant\mu<\frac{n-(n-1)q}{2-q}\big(\frac{1}{n-1}-q\big)\). Under the assumptions of Theorem \ref{thm-critical}, there exists a constant \(\delta>0\) depending on \(n,q,\mu\) such that 
\begin{align}\label{ineq-critical-0}
&\operatorname{div}\big(u^{\beta'}|\nabla u|^{(n-1)q^2-(n+1)q}l^{-\mu}\mathbf E\cdot\nabla u\big)
\geqslant\delta u^{\beta'}|\nabla u|^{(n-1)q^2-(n+1)q}l^{-\mu}|\mathbf E|^2
\end{align}
on \(Z^c\), where \(\beta'=-\frac{(n-1)^2q^2-(n^2-1)q+2}{n-2}\).
\end{proposition}

\begin{proof}
Under the selection from Proposition \ref{prop-id} and taking \(\varepsilon=0\), identity \eqref{id} reduces to
\begin{align}\label{cri-id}
&u^{-\beta'}|\nabla u|^{(n+1)q-(n-1)q^2}\big(u^{\beta'}|\nabla u|^{(n-1)q^2-(n+1)q}E_{ij}u^j\big),^i\nonumber\\
=~&E_{ij}E^{ij}-q[n+1-(n-1)q]\frac{E_{ij}u^jE^{ik}u_k}{|\nabla u|^2}+R_{ij}u^iu^j.
\end{align}
Introducing \(l^{-\mu}\) into the vector field, by the definition of \(l\) and Lemma \ref{lem-positivity}, we obtain
\begin{align*}
&u^{-\beta'}|\nabla u|^{(n+1)q-(n-1)q^2}l^\mu\big(u^{\beta'}|\nabla u|^{(n-1)q^2-(n+1)q}l^{-\mu}E_{ij}u^j\big),^i\\
=~&u^{-\beta'}|\nabla u|^{(n+1)q-(n-1)q^2}\big(u^{\beta'}|\nabla u|^{(n-1)q^2-(n+1)q}E_{ij}u^j\big),^i-\frac{\mu}{l}E_{ij}u^jl^i\\
=~&E_{ij}E^{ij}-\big[q[n+1-(n-1)q]+\mu(2-q)\frac{u^{-2_*(q)}|\nabla u|^{2-q}}{l}\big]\frac{1}{|\nabla u|^2}E_{ij}u^jE^{ik}u_k+R_{ij}u^iu^j\\
\geqslant~&\frac{n-1}{n}\big[[n-(n-1)q](\frac{1}{n-1}-q)-\mu(2-q)\big]E_{ij}E^{ij}.
\end{align*}
\end{proof}

The following two lemmas hold for all \(n\geqslant3\).

\begin{lemma}\label{h-pro}
Let \(h:\mathbb R_+\to\mathbb R_+\) and \(v:\mathbb R^n\to\mathbb R_+\) be two \(C^{2}\) functions. If \[\mathbf F:=\nabla^2 v+\frac{h''(v)}{h'(v)}\nabla v\otimes\nabla v-\frac{1}{n}(\Delta v+\frac{h''(v)}{h'(v)}|\nabla v|^2)g\equiv0,\]
then \(h(v)=c_1|x-x_0|^2+c_2\) for some \(x_0\in\mathbb R^n\) and positive constants \(c_1,c_2\).
\end{lemma}

\begin{proof}
For \(i\neq j\), \(F_{ij}=[h'(v)]^{-1}\partial_{ij}[h(v)]=0\), implying \(\partial_{ij}[h(v)]=0\). Therefore, for any \(1\leqslant i \leqslant n\), \(\partial_{i}[h(v)]\) depends only on \(x_i\), and it follows that for some \(h_i\in C^2(\mathbb R)\),
\[h(v)=h_i(x_i)+w_i(x_1, x_2, \dots,\hat{x}_i,\dots, x_n),\]
where \(w_i\) is independent of \(x_i\). Therefore, \(h(v)-\sum_{i=1}^n h_i(x_i)\) is constant.
    
For a fixed \(i\), we know that
\[F_{ii}=[h'(v)]^{-1}\partial_{ii}[h(v)]-\frac{1}{n}(\Delta v+\frac{h''(v)}{h'(v)}|\nabla v|^2)=0.\] 
It implies \(\partial_{ii}[h(v)]=\partial_{jj}[h(v)]\), that is, \(h_i''(x_i)=h_j''(x_j)\). Hence \(h_i''(x_i)\) is a constant independent of \(x_i\), i.e. \(h''(v)=2c_1\). And then, for some suitable positive constant \(c_2\), we have
\[h(v)=c_1|x-x_0|^2+c_2\]
 for some \(x_0\in\mathbb R^n.\)
\end{proof}

\begin{lemma}\label{critical-thm-pf}
Under the assumption of Theorem \ref{thm-critical}, if \(\mathbf E=0\), then \((M^n, g)\) is isometric to \(\mathbb R^n\) and \(u(x)\) must be of the form \eqref{solution}.
\end{lemma}

\begin{proof}
Noting that \(\nabla l=(2-q)u^{-2_*(q)}|\nabla u|^{-q}\mathbf E\cdot\nabla u=0\), it yields \(l = c_0 > 0\). Thus 
\begin{equation}\label{tidu u}
|\nabla u|^{2-q}=\big(c_0-\frac{2-q}{2^*(q)}u^{\frac{2-q}{(n-2)(1-q)}}\big)u^{2_*(q)}.
\end{equation}
Let \(h\) satisfy 
\begin{equation}\label{h-ass}
h'(u)=u^{-\frac{n}{n-2}}\big(c_0-\frac{2-q}{2^*(q)}u^{\frac{2-q}{(n-2)(1-q)}}\big)^{-\frac{q}{2-q}}.
\end{equation}
Based on \eqref{Meq-gradient} and \eqref{tidu u}, it holds
\[\frac{h''(u)}{h'(u)}=-\frac{n}{n-2}\frac 1 u-\frac{q}{n-(n-1)q}\frac{\Delta u}{|\nabla u|^2}.\]
Let \(v=h(u)\), it is obvious that \(v\) is a nontrivial function. Since \(\mathbf E=0\), from \eqref{cri-id} we immediately get \(\operatorname{Ric}(\nabla u,\nabla u)=0\), then \(\operatorname{Ric}(\nabla v,\nabla v)=0\). On the other hand,
\begin{align*}
\nabla^2 v-\frac{\Delta v}{n}g
=~&h''(u)\nabla u\otimes\nabla u+h'(u)\nabla^2 u-\frac 1 n (h''(u)|\nabla u|^2+h'(u)\Delta u)g\\
=~&h'(u)\big(\nabla^2 u+\frac{h''(u)}{h'(u)}\nabla u\otimes\nabla u-\frac{1}{n}(\Delta u+\frac{h''(u)}{h'(u)}|\nabla u|^2)g\big)\\
=~&h'(u)\mathbf E=0,
\end{align*}
which means 
\begin{equation}\label{M-con}
\nabla^2v=\frac{\Delta v}{n}g.
\end{equation}
Taking divergence of \eqref{M-con} yields \(\nabla\Delta v+\operatorname{Ric}(\nabla v,\cdot)=\frac{\nabla\Delta v}{n}\). Since \(\mathbf E=0\), then \(\operatorname{div}\mathbf E=0\), by \eqref{MEiji}, we have \(\operatorname{Ric}(\nabla v,\cdot)=0\). Thus \(\nabla\Delta v=0\). Therefore \(\nabla^2v=\frac{\Delta v}{n}g=\frac{c}{n}g\) for some constant \(c\). Combining \(\operatorname{Ric}(\nabla v,\nabla v)=0\), we have \((M^n, g)\) is isometric to \(\mathbb R^n\) by \cite[Theorem 2 (I, B)]{T1965}.

On the other hand, by the definition of \(h\) and Lemma \ref{h-pro}, we have
\[h(u)=c_1|x-x_0|^2+c_2.\] 
Thus \(h'(u)\nabla u=2c_1(x-x_0)\) for some \(c_1>0\). It follows from \eqref{h-ass} that 
\begin{equation}\label{tidu u2}
\nabla u=2c_1(x-x_0)u^{\frac{n}{n-2}}\big(c_0-\frac{2-q}{2^*(q)}u^{\frac{2-q}{(n-2)(1-q)}}\big)^{\frac{q}{2-q}}.
\end{equation}
By \eqref{tidu u} and \eqref{tidu u2}, it yields
\[u^{-\frac{2-q}{n-2}}=(2c_1)^{2-q}|x-x_0|^{2-q}\big(c_0-\frac{2-q}{2^*(q)}u^{\frac{2-q}{(n-2)(1-q)}}\big)^{q-1}.\] 
Rearranging the equation, that is,
\[\big(1+(2c_1)^{\frac{2-q}{q-1}}\frac{2-q}{2^*(q)}|x-x_0|^{\frac{2-q}{q-1}}\big)u^{\frac{2-q}{(n-2)(1-q)}}=(2c_1)^{\frac{2-q}{q-1}}|x-x_0|^{\frac{2-q}{q-1}}c_0.\]
Therefore, \(u=\big(\frac{\frac{2^*(q)}{2-q}c_0}{1+(2c_1)^{\frac{2-q}{1-q}}\frac{2^*(q)}{2-q}|x-x_0|^{\frac{2-q}{1-q}}}\big)^{\frac{(n-2)(1-q)}{2-q}}\). Choosing \(c_1=\big(\frac{2-q}{2^*(q)}\big)^{\frac{1-q}{2-q}}\frac\lambda 2\), the constant \(c_0\) can be solved from equation \eqref{Meq-gradient}. After which, \(u\) must be of the form \eqref{solution}.
\end{proof}

An integral estimate is required for further analysis.

\begin{lemma}\label{prop-critical-est}
Let \((s,t)\in \mathbb{R}\times \mathbb R\) satisfies
\begin{align*}
\begin{cases}
-t\leqslant s<[2^*(q)-1]\frac{t}{q},&0<t\leqslant q,\\
-t\leqslant s<\frac{2-t}{n-2}\Big(n+\frac{q}{1-q}\Big)-1,&q<t\leqslant2.
\end{cases}
\end{align*}
Then \[\int_{B_R}u^s|\nabla u|^t\leqslant CR^{n-\frac{1-q}{2-q}(n-2)s-\frac{(1-q)n+q}{2-q}t}.\] 
\end{lemma}

\begin{proof} Let \(-1<\beta_1<0\) and sufficiently large \(\gamma>0\) be independent of \(R\). 
Testing equation \eqref{Meq-gradient} with \(u^{\beta_1}\eta^\gamma\) and applying Young’s inequality yields
\begin{align}\label{ineq-critical-3}
&\int_{M^n} u^{\beta_1+2^*(q)-1}|\nabla u|^q\eta^\gamma-\beta_1\int_{M^n}u^{\beta_1-1}|\nabla u|^2\eta^\gamma=\gamma\int_{M^n}u^{\beta_1}\eta^{\gamma-1}\nabla u\cdot\nabla\eta\nonumber\\
\leqslant~&\frac{1}{2}\int_{M^n} u^{\beta_1+2^*(q)-1}|\nabla u|^q\eta^\gamma-\frac{\beta_1}{2}\int_{M^n}u^{\beta_1-1}|\nabla u|^2\eta^\gamma+C\int_{M^n}\eta^{\gamma-\frac{1}{c_3}}|\nabla\eta|^{\frac{1}{c_3}},
\end{align}
where \(c_3=\frac{2-q}{[n+(n-2)\beta_1](1-q)+2}\). Taking \(\gamma=\frac{1}{c_3}\), then using the Bishop-Gromov volume comparison theorem, it follows from \eqref{ineq-critical-3} that
\begin{equation}\label{ineq-critical-4}
\int_{B_R}u^{\beta_1+2^*(q)-1}|\nabla u|^q+\int_{B_R}u^{\beta_1-1}|\nabla u|^2\leqslant CR^{\frac{(n-2)[1-(1-q)\beta_1]}{2-q}}.
\end{equation}
Let \((s_1,t_1)\in S_1:=\big\{-t_1\leqslant s_1<-\frac{t_1}{2},~0<t_1\leqslant2\big\}\), and choose \(\beta_1=\frac{2s_1}{t_1}+1\). According to the Bishop-Gromov volume comparison theorem, by H\"{o}lder inequality and \eqref{ineq-critical-4} it holds 
\begin{align}\label{ineq-critical-5}
\int_{M^n}u^{s_1}|\nabla u|^{t_1}\eta^\gamma
\leq~&\big(\int_{M^n}u^{\frac{2s_1}{t_1}}|\nabla u|^2\eta^\gamma\big)^{\frac{t_1}{2}} \big(\int_{M^n}\eta^\gamma\big)^{\frac{2-t_1}{2}}\nonumber\\
\leq~& CR^{n-\frac{1-q}{2-q}(n-2)s_1-\frac{(1-q)n+q}{2-q}t_1}.    
\end{align}
 
Similarly, for
\[(s_2,t_2)\in S_2:=\big\{[2^*(q)-2]\frac{t_2}{q}\leqslant s_2<[2^*(q)-1]\frac{t_2}{q},~0<t_2\leqslant q\big\},\]
taking \(\beta_1=\frac{qs_2}{t_2}+1-2^*(q)\), the Bishop-Gromov volume comparison theorem yields
\begin{align}\label{ineq-critical-6}
\int_{M^n}u^{s_2}|\nabla u|^{t_2}\eta^\gamma
\leq~&\big(\int_{M^n}u^{\frac{qs_2}{t_2}}|\nabla u|^q\eta^\gamma\big)^{\frac{t_2}{q}}\big(\int_{M^n}\eta^\gamma\big)^{\frac{q-t_2}{q}}\nonumber\\
\leq~&CR^{n-\frac{1-q}{2-q}(n-2)s_2-\frac{(1-q)n+q}{2-q}t_2}.
\end{align}
Now, let \((s,t)\) belong to the convex hull of \(S_1\cup S_2\), By H\"older's inequality, it follows from \eqref{ineq-critical-5} and \eqref{ineq-critical-6} that
\[\int_{B_R}u^s|\nabla u|^t\leqslant CR^{n-\frac{1-q}{2-q}(n-2)s-\frac{(1-q)n+q}{2-q}t}.\]
This completes the proof.
\end{proof}

Now, we are ready to finish the proof of Theorem \ref{thm-critical} based on the above results.

\textbf{Proof of Theorem \ref{thm-critical}.}
Testing \eqref{ineq-critical-0} by \(\eta^2\) yields
\begin{align}\label{ineq-critical-1}
&\int_{Z^c} u^{\beta'}|\nabla u|^{(n-1)q^2-(n+1)q}l^{-\mu}|\mathbf E|^2\eta^2\nonumber\\
\leqslant~&C\int_{\partial Z \cap B_{2R}}u^{\beta'}|\nabla u|^{(n-1)q^2-(n+1)q}l^{-\mu}|\mathbf E\cdot\nabla u\otimes\nu|\eta^2\nonumber\\   
~&+C\int_{Z^c\cap B_{2R}}u^{\beta'}|\nabla u|^{(n-1)q^2-(n+1)q}l^{-\mu}|\mathbf E\cdot\nabla u\otimes\eta|.
\end{align}
    
Following the discussion of inequality \eqref{ineq-sub-1} in Section \ref{sec-subcritical}, which corresponds to the case where \(d=(n-1)q^2-(n+1)q\), we can show that the first term on the RHS of inequality \eqref{ineq-critical-1} vanishes. 

Next, we estimate the second term on the RHS of \eqref{ineq-critical-1} using the mean value inequality
\begin{align*}
&\int_{Z^c\cap B_{2R}}u^{\beta'}|\nabla u|^{(n-1)q^2-(n+1)q}l^{-\mu}|\mathbf E\cdot\nabla u\otimes\eta|\eta\\
\leqslant~&\int_{Z^c\cap B_{2R}} u^{\beta'}|\nabla u|^{(n-1)q^2-(n+1)q}l^{-\mu}\big(\varepsilon|\mathbf E|^2\eta^2+C|\nabla u|^2|\nabla \eta|^2\big).
\end{align*}
Substituting this into \eqref{ineq-critical-1} and using \(l\geqslant Cu^{\frac{2-q}{(n-2)(1-q)}}\), it yields
\begin{align}\label{ineq-critical-2}
\int_{Z^c\cap B_R} u^{\beta'}|\nabla u|^{(n-1)q^2-(n+1)q}l^{-\mu}|\mathbf E|^2
\leqslant~&CR^{-2}\int_{Z^c\cap B_{2R}} u^{\beta'}|\nabla u|^{(n-1)q^2-(n+1)q+2}l^{-\mu}\nonumber\\
\leqslant~&CR^{-2}\int_{Z^c\cap B_{2R}} u^s|\nabla u|^t,   
\end{align}
where \(s=\beta'-\frac{2-q}{(n-2)(1-q)}\mu\), \(t=(n-1)q^2-(n+1)q+2\).
    
Obviously, \(t-q=(n-1)q^2-(n+2)q+2>\frac{n-3}{n-1}\geqslant 0\) holds. We divide dimension \(n\) into two cases.
    
\textbf{Case 1. \(n=3\) or \(4\).} Take \(\mu=0\). It is easy to verify
\begin{align*}
&s+t=\frac{(2-q)[(n-1)q+n-3]}{n-2}\geqslant0,\\
&\frac{2-t}{n-2}\big(n+\frac{q}{1-q}\big)-1-s=\frac{1}{n-2}\big(\frac{2}{1-q}+1-(n-1)(1-q)\big)>\frac{4-n}{n-2}\geqslant0,\\
&n-\frac{1-q}{2-q}(n-2)s-\frac{(1-q)n+q}{2-q}t=(n-1)q+1.
\end{align*}
From Lemma \ref{prop-critical-est} and the Bishop-Gromov volume comparison theorem, inequality \eqref{ineq-critical-2} becomes
\[\int_{Z^c\cap B_R} u^{\beta'}|\nabla u|^{(n-1)q^2-(n+1)q}l^{-\mu}|\mathbf E|^2\leqslant CR^{(n-1)q-1}.\]

\textbf{Case 2. \(n=5\).} In this case, \(0<q<\frac 1 4\). Choosing \(\mu=\frac{5-4q}{2-q}(\frac 1 4-q)-\varepsilon\), we obtain
\begin{align*}
&\lim_{\varepsilon\to0^+}(s+t)=\lim_{\varepsilon\to0^+}\frac{2-q}{3}(4q+2-\frac{\mu}{1-q})=\frac{16q^3-56q^2+32q+11}{12(1-q)}>0,\\
&\lim_{\varepsilon\to0^+}\big(\frac{2-t}{n-2}(n+\frac{q}{1-q})-1-s\big)=\lim_{\varepsilon\to0^+}\frac{(2-q)\mu-4q^2+7q-1}{3(1-q)}=\frac{4q+1}{12(1-q)}>0,\\
&n-\frac{1-q}{2-q}(n-2)s-\frac{(1-q)n+q}{2-q}t=\mu+4q+1=2-\frac{3(1-4q)}{4(2-q)}-\varepsilon.
\end{align*}
Using Lemma \ref{prop-critical-est} and the Bishop-Gromov volume comparison theorem again, inequality \eqref{ineq-critical-2} becomes
\[\int_{Z^c\cap B_R} u^{\beta'}|\nabla u|^{(n-1)q^2-(n+1)q}l^{-\mu}|\mathbf E|^2\leqslant CR^{-\frac{3(1-4q)}{4(2-q)}-\varepsilon}.\]

In both cases, \(\mathbf E=0\) holds in \(Z^c\) by letting \(R\to\infty\). For \(x\in\mathring Z\), \(\nabla u(x)=0\) yields \(\mathbf E=0\). By the continuity of \(\mathbf E\), we immediately conclude that \(\mathbf E=0\) on \(M\). Together with Lemma \ref{critical-thm-pf}, the proof is complete. 
\qed

\textbf{Acknowledgments.} The authors would like to thank Professor Xi-Nan Ma for the constant encouragement in this paper. Jingbo Dou is supported by the National Natural Science Foundation of China (Grant No. 12471109), Youth Innovation Team of Shaanxi Universities and the Fundamental Research Funds for the Central Universities (Grant No. GK202307001, GK202402004), Tian Wu is supported by Anhui Postdoctoral Scientific Research Program Foundation (Grant No. 2025B1055) and the Fundamental Research Funds for the Central Universities (Grant No. WK0010250106). Hua Zhu is supported by the National Natural Science Foundation of China (Grant No. 12501273) and the Natural Science Foundation of Southwest University of Science and Technology (Grant No. 25zx7153)


\footnotesize{
    Welcome contact us:
    \begin{itemize}
        \item Jingbo Dou, School of Mathematics and Statistics, Shaanxi Normal University, Xi'an, Shaanxi, 710119, People's Republic of China. Email: \emph{jbdou@snnu.edu.cn}
        \item Benfeng Shi, School of Mathematics and Statistics, Shaanxi Normal University, Xi'an, Shaanxi, 710119, People's Republic of China. Email: \emph{shibenfeng@snnu.edu.cn}
        \item Tian Wu, School of Mathematical Sciences, University of Science and Technology of China, Hefei, Anhui, 230026, People's Republic of China. Email: \emph{wt1997@ustc.edu.cn}
        \item Hua Zhu, School of Mathematical and Physics, Southwest University of Science and Technology, Mianyang, Sichuan, 621010, People's Republic of China. Email: \emph{zhuhmaths@mail.ustc.edu.cn}
    \end{itemize}
}

\end{document}